\documentclass[10pt,twoside]{article} 

\date{} 

\usepackage{amssymb,amsmath,latexsym,amsthm,amsfonts}

\newcommand{\N}{\mbox{$I \kern -4pt N$}}      

\newcommand{\Z}{\mbox{$Z \kern -7.5pt Z$}}     

\newcommand{\Q}{\mbox{$Q \kern -8pt I$}}      

\newcommand{\R}{{\bf R}}
\newcommand{\C}{\mbox{$C \kern -8pt I$}}      

\newcommand{\vp}{\varphi}
\newcommand{\dxt}{dx\,dt}
\newcommand{\tell}{\tilde{\ell}}
\newcommand{\tpsi}{\widetilde\psi}

\font\dsrom=dsrom10 scaled 1200
\def \1{\textrm{\dsrom{1}}}

\def\om       {\omega}

\def\Om       {\Omega}

\newtheorem{theorem}{Theorem}[section]

\newtheorem{lemma}[theorem]{Lemma}
\newtheorem{proposition}{Proposition}

\theoremstyle{definition}

\newtheorem{remark}{Remark}

\numberwithin{equation}{section}

\begin{document} 

\title{ Local controllability of the N-dimensional Boussinesq system with N-1 scalar controls in an arbitrary control domain}

\author{Nicol\'as Carre\~no
\footnote{Universit\'e Pierre et Marie Curie, UMR 7598, Laboratoire Jacques-Louis Lions, F-75005, Paris, France;
{\bf ncarreno@ann.jussieu.fr}
  }
}

\maketitle

\begin{abstract} 
 In this paper we deal with the local exact controllability to a particular class of trajectories of the $N-$dimensional Boussinesq system with internal controls having $2$ vanishing components. The main novelty of this work is that no condition is imposed on the control domain.
\end{abstract} 


\section{Introduction}

Let $\Om$ be a nonempty bounded connected open subset of $\R^N$ ($N=2$ or $3$) of class $C^{\infty}$.
Let $T>0$ and let $\om\subset\Om$ be a (small) nonempty open subset which is the control domain. We will 
use the notation $Q=\Om\times(0,T)$ and $\Sigma=\partial\Om\times (0,T)$. 

We will be concerned with the following controlled Boussinesq system:
\begin{equation}\label{eq:Bouss}
\left\lbrace \begin{array}{ll}
    y_t - \Delta y + (y\cdot \nabla)y + \nabla p = v\1_{\om} + \theta\,e_N& \mbox{ in }Q, \\
    \theta_t - \Delta \theta + y\cdot \nabla\theta = v_0\1_{\om} & \mbox{ in }Q, \\
    \nabla\cdot y = 0 & \mbox{ in }Q, \\
    y = 0,\, \theta=0 & \mbox{ on }\Sigma, \\
    y(0) = y^0,\, \theta(0)=\theta^0 & \mbox{ in }\Om,
 \end{array}\right.
\end{equation}
where
$$
e_N = \left\lbrace
\begin{array}{ll}
(0,1)& \mbox{ if }N=2,\\
(0,0,1)& \mbox{ if }N=3
\end{array}\right.
$$ 
stands for the gravity vector field, $y=y(x,t)$ represents the velocity of the particules of the fluid, $\theta=\theta(x,t)$ their temperature and $(v_0,v)=(v_0, v_1,\dots,v_N)$ stands for the control which acts over the set $\om$.

Let us recall the definition of some usual spaces in the context of incompressible fluids:
$$V=\{y\in H^1_0(\Om)^N:\,\nabla\cdot y=0 \mbox{ in }\Om  \}$$
and
$$H=\{y\in L^2(\Om)^N:\,\nabla\cdot y=0 \mbox{ in }\Om,\,y\cdot n =0 \mbox{ on }\partial\Om\}.$$

This paper concerns the local exact controllability to the trajectories of system \eqref{eq:Bouss} at time $t=T$ with a reduced number of controls. To introduce this concept, let us consider $(\bar y, \bar \theta)$ (together with some pressure $\bar p$) a trajectory of the following uncontrolled Boussinesq system:
\begin{equation}\label{eq:ncBouss}
\left\lbrace \begin{array}{ll}
    \bar y_t - \Delta \bar y + (\bar y\cdot \nabla)\bar y + \nabla \bar p = \bar \theta\,e_N & \mbox{ in }Q, \\
    \bar \theta_t - \Delta \bar \theta + \bar y\cdot \nabla \bar \theta = 0 & \mbox{ in }Q, \\
    \nabla\cdot \bar y = 0 & \mbox{ in }Q, \\
    \bar y = 0,\, \bar \theta=0 & \mbox{ on }\Sigma, \\
    \bar y(0) = \bar y^0,\, \bar \theta(0)=\bar \theta^0 & \mbox{ in }\Om.
 \end{array}\right.
\end{equation}
We say that the local exact controllability to the trajectories $(\bar y, \bar \theta)$ holds if there exists a number $\delta >0$ such that if $\|(y^0,\theta^0)-(\bar y^0,\bar\theta^0)\|_X\leq \delta$ ($X$ is an appropriate Banach space), there exist controls $(v_0,v)\in L^2(\om\times (0,T))^{N+1}$ such that the corresponding solution $(y,\theta)$ to system \eqref{eq:Bouss} matches $(\bar y,\bar\theta)$ at time $t=T$, i.e.,
\begin{equation}\label{conditionT}
y(T)=\bar y(T) \mbox{ and }\theta(T)=\bar\theta(T)\mbox{ in }\Om.
\end{equation}

The first results concerning this problem were obtained in \cite{FurIma98} and \cite{FurIma99}, with $N+1$ scalar controls acting in the whole boundary of $\Om$ and with $N+1$ scalar controls acting in $\om$ when $\Om$ is a torus, respectively. Later, in \cite{Gue}, the author proved the local exact controllability for less regular trajectories $(\bar y, \bar \theta)$ in an open bounded set and for an arbitrary control domain. Namely, the trajectories were supposed to satisfy
\begin{equation}\label{regularity}
(\bar y, \bar \theta) \in L^{\infty}(Q)^{N+1},\, (\bar y_t, \bar \theta_t)\in L^2(0,T;L^r(\Om))^{N+1},
\end{equation}
with $r>1$ if $N=2$ and $r>6/5$ if $N=3$.

In \cite{E&S&O&P-N-1}, the authors proved that local exact controllability can be achieved with $N-1$ scalar controls acting in $\om$ when $\overline \om$ intersects the boundary of $\Om$ and \eqref{regularity} is satisfied. More precisely, we can find controls $v_0$ and $v$, with $v_N \equiv 0$ and $v_k\equiv 0$ for some $k<N$ ($k$ is determined by some geometric assumption on $\om$, see \cite{E&S&O&P-N-1} for more details), such that the corresponding solution to \eqref{eq:Bouss} satisfies \eqref{conditionT}. 

In this work, we remove this geometric assumption on $\om$ and consider a target trajectory of the form $(0, \bar p, \bar \theta )$, i.e.,  
\begin{equation}\label{eq:ncBouss0}
\left\lbrace \begin{array}{ll}
    \nabla \bar p = \bar \theta\,e_N & \mbox{ in }Q, \\
    \bar \theta_t - \Delta \bar \theta  = 0 & \mbox{ in }Q, \\
    \bar \theta=0 & \mbox{ on }\Sigma, \\
    \bar \theta(0)=\bar \theta^0 & \mbox{ in }\Om,
 \end{array}\right.
\end{equation}
where we assume
\begin{equation}\label{regularity0}
 \bar \theta \in L^{\infty}(0,T;W^{3,\infty}(\Om))\mbox{ and }\nabla\bar\theta_t \in L^{\infty}(Q)^N.
\end{equation}

The main result of this paper is given in the following theorem.
\begin{theorem}\label{teo:controlBouss}
Let $i<N$ be a positive integer and $(\bar p, \bar\theta)$ a solution to \eqref{eq:ncBouss0} satisfying \eqref{regularity0}. Then, for every $T>0$ and $\om\subset\Om$, there exists $\delta>0$ such that for every $(y^0,\theta^0)\in V\times H^1_0(\Om)$ satisfying
$$\|(y^0,\theta^0)-(0,\bar\theta^0)\|_{V\times H^1_0}\leq \delta,$$
we can find controls $v^0\in L^2(\om\times (0,T))$ and $v\in L^2(\om\times (0,T))^N$, with $v_i\equiv 0$ and $v_N\equiv 0$, such that the corresponding solution to \eqref{eq:Bouss} satisfies \eqref{conditionT}, i.e.,
\begin{equation}\label{conditionT0}
y(T)=0 \mbox{ and }\theta(T)=\bar\theta(T)\mbox{ in }\Om.
\end{equation}
\end{theorem}

\begin{remark}
Notice that when $N=2$ we only need to control the temperature equation.
\end{remark}

\begin{remark}
It would be interesting to know if the local controllability to the trajectories with $N-1$ scalar controls holds for $\bar y\neq 0$ and $\om$ as in Theorem \ref{teo:controlBouss}. However, up to our knowledge, this is an open problem even for the case of the Navier-Stokes system. 
\end{remark}

\begin{remark}
One could also try to just control the movement equation, that is, $v_0 \equiv 0$ in \eqref{eq:Bouss}. However, this system does not seem to be controllable. To justify this, let us consider the control problem
\begin{equation*}
\left\lbrace \begin{array}{ll}
    y_t - \Delta y + (y\cdot \nabla)y + \nabla p = v\1_{\om} + \theta\,e_N& \mbox{ in }Q, \\
    \theta_t - \Delta \theta + y\cdot \nabla\theta = 0 & \mbox{ in }Q, \\
    \nabla\cdot y = 0 & \mbox{ in }Q, \\
    y = 0,\, \nabla\theta\cdot n=0 & \mbox{ on }\Sigma, \\
    y(0) = y^0,\, \theta(0)=\theta^0 & \mbox{ in }\Om;
 \end{array}\right.
\end{equation*}
where we have homogeneous Neumann boundary conditions for the temperature. Integrating in $Q$, integration by parts gives
\begin{equation*}
 \int\limits_{\Om}\theta(T)\,dx = \int\limits_{\Om} \theta_0\,dx,
\end{equation*}
so we can not expect in general null controllability.
\end{remark}

Some recent works have been developed in the controllability problem with reduced number of controls. For instance, in \cite{CorGue} the authors proved the null controllability for the Stokes system with $N-1$ scalar controls, and in \cite{CarrGue} the local null controllability was proved for the Navier-Stokes system with the same number of controls. 

The present work can be viewed as an extension of \cite{CarrGue}. To prove Theorem \ref{teo:controlBouss} we follow a standard approach introduced in \cite{FurIma} and \cite{OlegN-S} (see also \cite{E&S&O&P}). We first deduce a null controllability result for the linear system
\begin{equation}\label{eq:linearBouss}
\left\lbrace \begin{array}{ll}
    y_t - \Delta y  + \nabla p = f + v\1_{\om} + \theta\,e_N& \mbox{ in }Q, \\
    \theta_t - \Delta \theta + y\cdot \nabla\bar\theta = f_0 + v_0\1_{\om} & \mbox{ in }Q, \\
    \nabla\cdot y = 0 & \mbox{ in }Q, \\
    y = 0,\, \theta=0 & \mbox{ on }\Sigma, \\
    y(0) = y^0,\, \theta(0)=\theta^0 & \mbox{ in }\Om,
 \end{array}\right.
\end{equation}
where $f$ and $f_0$ will be taken to decrease exponentially to zero in $t=T$. 

The main tool to prove this null controllability result for system \eqref{eq:linearBouss} is a suitable Carleman estimate for the solutions of its adjoint system, namely,
\begin{equation}\label{eq:adjointBouss}
\left\lbrace \begin{array}{ll}
    -\vp_t - \Delta \vp  + \nabla \pi = g - \psi\nabla\bar\theta& \mbox{ in }Q, \\
    -\psi_t - \Delta \psi = g_0 + \vp_N & \mbox{ in }Q, \\
    \nabla\cdot \vp = 0 & \mbox{ in }Q, \\
    \vp = 0,\, \psi=0 & \mbox{ on }\Sigma, \\
    \vp(T) = \vp^T,\, \psi(T)=\psi^T & \mbox{ in }\Om,
 \end{array}\right.
\end{equation}
 where $g\in L^2(Q)^N$, $g_0\in L^2(Q)$, $\vp^T\in H$ and $\psi^T\in L^2(\Om)$. In fact, this inequality is of the form
\begin{multline}\label{eq:0}
\iint\limits_{Q}  \widetilde \rho_1(t)(|\vp|^2 + |\psi|^2) dx\,dt \\
\leq  C \left( \iint\limits_{Q}  \widetilde \rho_2(t)(|g|^2 +|g_0|^2 ) dx\,dt 
+   \int\limits_0^T\int\limits_{\om} \widetilde \rho_3(t) |\vp_j|^2 dx\,dt 
 +   \int\limits_0^T\int\limits_{\om} \widetilde \rho_4(t) |\psi|^2 dx\,dt \right),
\end{multline}
if $N=3$, and of the form
\begin{multline*}
\iint\limits_{Q}  \widetilde \rho_1(t)(|\vp|^2 + |\psi|^2) dx dt 
\leq  C \left( \iint\limits_{Q}  \widetilde \rho_2(t)(|g|^2 +|g_0|^2 ) dx dt 
 +   \int\limits_0^T\int\limits_{\om} \widetilde \rho_4(t) |\psi|^2 dx dt \right),
\end{multline*}
if $N=2$, where $j=1$ or 2 and $\widetilde\rho_k(t)$ are positive smooth weight functions (see inequalities \eqref{eq:Carleman3} and \eqref{eq:Carleman2} below). From these estimates, we can find a solution $(y,\theta,v,v_0)$ of \eqref{eq:linearBouss} with the same decreasing properties as $f$ and $f_0$. In particular, $(y(T),\theta(T))=(0,0)$ and $v_i=v_N=0.$ 

We conclude the controllability result for the nonlinear system by means of an inverse mapping theorem.

This paper is organized as follows. In section 2, we prove a Carleman inequality of the form \eqref{eq:0} for system \eqref{eq:adjointBouss}. In section 3, we deal with the null controllability of the linear system \eqref{eq:linearBouss}. Finally, in section 4 we give the proof of Theorem \ref{teo:controlBouss}.


\section{Carleman estimate for the adjoint system}

In this section we will prove a Carleman estimate for the adjoint system \eqref{eq:adjointBouss}. In order to do so, we are going to introduce some weight functions. Let $\om_0$ be a nonempty open subset of $\R^N$ such that $\overline{\om_0}\subset \om$ and $\eta\in C^2(\overline{\Om})$ such that
\begin{equation}
 |\nabla \eta|>0 \mbox{ in }\overline{\Om}\setminus\om_0,\, \eta>0 \mbox{ in }\Om \mbox{ and } \eta \equiv 0 \mbox{ on }\partial\Om.
\end{equation}
The existence of such a function $\eta$ is given in \cite{FurIma}.
Let also $\ell\in C^{\infty}([0,T])$ be a positive function satisfying 
\begin{equation}
\begin{split}
&\ell(t) = t \quad \forall t \in [0,T/4],\,\ell(t) = T-t \quad \forall t \in [3T/4,T],\\
& \ell(t)\leq \ell(T/2),\, \forall t\in [0,T].
\end{split}
\end{equation}

Then, for all $\lambda\geq 1$ we consider the following weight functions:
\begin{equation}\label{pesos}
\begin{split}
&\alpha(x,t) = \dfrac{e^{2\lambda\|\eta\|_{\infty}}-e^{\lambda\eta(x)}}{\ell^{8}(t)},\, \xi(x,t)=\dfrac{e^{\lambda\eta(x)}}{\ell^8(t)},\\
&\alpha^*(t) = \max_{x\in\overline{\Om}} \alpha(x,t),\, \xi^*(t) = \min_{x\in\overline{\Om}} \xi(x,t),\\
&\widehat\alpha(t) = \min_{x\in\overline{\Om}} \alpha(x,t),\, \widehat\xi(t) = \max_{x\in\overline{\Om}} \xi(x,t).
\end{split}
\end{equation}

Our Carleman estimate is given in the following proposition.
\begin{proposition}\label{prop:Carleman3}
Assume $N=3$, $\om\subset\Om$ and $(\bar p, \bar \theta)$ satisfies \eqref{regularity0}. There exists a constant $\lambda_0$, such that for any $\lambda\geq\lambda_0$ there exist two constants $C(\lambda)>0$ and $s_0(\lambda)>0$ such that for any $j\in\{1,2\}$, any $g\in L^2(Q)^3$, any $g_0\in L^2(Q)$, any $\vp^T \in H$ and any $\psi^T \in L^2(\Om)$, the solution of \eqref{eq:adjointBouss} satisfies
\begin{multline}\label{eq:Carleman3}
s^4\iint\limits_{Q}  e^{-5s\alpha^*}  (\xi^*)^4 |\vp|^2 dx\,dt + s^5\iint\limits_{Q}  e^{-5s\alpha^*}  (\xi^*)^5 |\psi|^2 dx\,dt \\
\leq  C \left( \iint\limits_{Q}  e^{-3s \alpha^*}(|g|^2 +|g_0|^2 ) dx\,dt 
+ s^7  \int\limits_0^T\int\limits_{\om}e^{-2s\widehat\alpha - 3s\alpha^*}(\widehat\xi)^7|\vp_j|^2 dx\,dt \right. \\
\left. +  s^{12}  \int\limits_0^T\int\limits_{\om}e^{-4s\widehat\alpha - s\alpha^*}(\widehat\xi)^{49/4}|\psi|^2 dx\,dt \right)
\end{multline}
for every $s\geq s_0$.
\end{proposition}
 
For the sake of completeness, let us also state this result for the $2$-dimensional case.
\begin{proposition}\label{prop:Carleman2}
Assume $N=2$, $\om\subset\Om$ and $(\bar p, \bar \theta)$ satisfies \eqref{regularity0}. There exists a constant $\lambda_0$, such that for any $\lambda\geq\lambda_0$ there exist two constants $C(\lambda)>0$ and $s_0(\lambda)>0$ such that for any $g\in L^2(Q)^2$, any $g_0\in L^2(Q)$, any $\vp^T \in H$ and any $\psi^T \in L^2(\Om)$, the solution of \eqref{eq:adjointBouss} satisfies
\begin{multline}\label{eq:Carleman2}
s^4\iint\limits_{Q}  e^{-5s\alpha^*}  (\xi^*)^4 |\vp|^2 dx\,dt + s^5\iint\limits_{Q}  e^{-5s\alpha^*}  (\xi^*)^5 |\psi|^2 dx\,dt \\
\leq  C \left( \iint\limits_{Q}  e^{-3s \alpha^*}(|g|^2 +|g_0|^2 ) dx\,dt 
+  s^{12}  \int\limits_0^T\int\limits_{\om}e^{-4s\widehat\alpha - s\alpha^*}(\widehat\xi)^{49/4}|\psi|^2 dx\,dt \right)
\end{multline}
for every $s\geq s_0$.
\end{proposition}

To prove Proposition \ref{prop:Carleman3} we will follow the ideas of \cite{CorGue} and \cite{E&S&O&P-N-1} (see also \cite{CarrGue}). An important point in the proof of the Carleman inequality established in \cite{CorGue} is that the laplacian of the pressure in the adjoint system is zero. In \cite{CarrGue}, a decomposition of the solution was made, so that we can essentially concentrate in a solution where the laplacian of the pressure is zero. For system \eqref{eq:adjointBouss} this will not be possible because of the coupling term $\psi\nabla\bar\theta$. However, under hypothesis \eqref{regularity0} we can follow the same ideas to obtain \eqref{eq:Carleman3}. All the details are given below.

\subsection{Technical results}

Let us present now the technical results needed to prove Carleman inequalities \eqref{eq:Carleman3} and \eqref{eq:Carleman2}. The first of these results is a Carleman inequality for parabolic equations with nonhomogeneous boundary conditions proved in \cite{ImaPuelYam}. Consider the equation
\begin{equation}\label{eq:heatnonhom}
 u_t - \Delta u = F_0 + \sum_{j=1}^N \partial_j F_j \mbox{ in }Q,
\end{equation}
where $F_0,F_1,\dots,F_N\in L^2(Q)$. We have the following result.

\begin{lemma}\label{lemma:Cnonhom}
There exists a constant $\widehat\lambda_0$ only depending on $\Om$, $\om_0$, $\eta$ and $\ell$ such that for any $\lambda>\widehat\lambda_0$ there exist two constants $C(\lambda)>0$ and $\widehat{s}(\lambda)$, such that for every $s\geq \widehat{s}$ and every $u\in L^2(0,T;H^1(\Om))\cap H^1(0,T;H^{-1}(\Om))$ satisfying \eqref{eq:heatnonhom}, we have
\begin{multline}\label{eq:Cnonhom}
 \dfrac{1}{s}\iint\limits_Q e^{-2s\alpha} \dfrac{1}{\xi}|\nabla u|^2 \dxt 
+ s\iint\limits_Q e^{-2s\alpha} \xi |u|^2 \dxt 
\leq C\left(  s\int\limits_0^T\int\limits_{\om_0}e^{-2s\alpha}\xi|u|^2 \dxt \right.\\
\qquad\qquad\qquad +s^{-1/2} \left\|e^{-s\alpha}\xi^{-1/4}u\right\|^2_{H^{\frac{1}{4},\frac{1}{2}}(\Sigma)} 
+ s^{-1/2} \left\|e^{-s\alpha}\xi^{-1/8}u\right\|^2_{L^2(\Sigma)} \\
\left. 
+ s^{-2}\iint\limits_Q e^{-2s\alpha}\xi^{-2}|F_0|^2\dxt 
 +\sum_{j=1}^N\iint\limits_Q e^{-2s\alpha}|F_j|^2 \dxt \right).
\end{multline}
\end{lemma}

Recall that
$$\|u\|_{H^{\frac{1}{4},\frac{1}{2}}(\Sigma)}=\left(\|u\|^2_{H^{1/4}(0,T;L^2(\partial\Om))} + \|u\|^2_{L^{2}(0,T;H^{1/2}(\partial\Om))} \right)^{1/2}.$$

The next technical result is a particular case of Lemma 3 in \cite{CorGue}.

\begin{lemma}\label{lemma1}
There exists a constant $\widehat\lambda_1$ such that for any $\lambda\geq \widehat\lambda_1$ there exists $C>0$ depending only on $\lambda$, $\Om$, $\om_0$, $\eta$ and $\ell$ such that, for every $T>0$ and every $u\in L^2(0,T;H^1(\Om))$,
\begin{multline}\label{eq:lemma1}
 s^3\iint\limits_Q e^{-2s\alpha}\xi^3|u|^2\dxt \\
\leq C \left( s \iint\limits_Q e^{-2s\alpha}\xi|\nabla u|^2\dxt + s^3 \int\limits_0^T\int\limits_{\om_0} e^{-2s\alpha}\xi^3|u|^2\dxt \right),
\end{multline}
for every $s\geq C$.
\end{lemma}

The next lemma is an estimate concerning the Laplace operator:

\begin{lemma}\label{lemma2}
 There exists a constant $\widehat\lambda_2$ such that for any $\lambda\geq \widehat\lambda_2$ there exists $C>0$ depending only on $\lambda$, $\Om$, $\om_0$, $\eta$ and $\ell$ such that, for every $u\in L^2(0,T;H^1_0(\Om))$,
\begin{multline}\label{eq:lemma2}
s^6\iint\limits_{Q}e^{-2s\alpha}\xi^6|u|^2 \dxt +  s^4\iint\limits_{Q}e^{-2s\alpha}\xi^4|\nabla u|^2 \dxt \\
\leq C \left( s^3\iint\limits_{Q}e^{-2s\alpha}\xi^3|\Delta u|^2 \dxt + s^6\int\limits_0^T\int\limits_{\om_0}e^{-2s\alpha}\xi^6| u|^2 \dxt \right),
\end{multline}
for every $s\geq C$.
\end{lemma}

Inequality \eqref{eq:lemma2} comes from the classical result in \cite{FurIma} for parabolic equations applied to the laplacian with parameter $s/\ell^8(t)$. Then, multiplying by \\$\exp(-2se^{2\lambda\|\eta\|_{\infty}}/\ell^8(t))$ and integrating in $(0,T)$ we obtain \eqref{eq:lemma2}. Details can be found in \cite{CorGue} or \cite{CarrGue}. 

The last technical result concerns the regularity of the solutions to the Stokes system that can be found in \cite{Lady} (see also \cite{Temam}).

\begin{lemma}
For every $T>0$ and every $F\in L^2(Q)^N$, there exists a unique solution 
$$u\in L^2(0,T;H^2(\Om)^N) \cap H^1(0,T;H)$$ 
to the Stokes system
\begin{equation*}
\left\lbrace \begin{array}{ll}
    u_t - \Delta u + \nabla p = F  & \mbox{ in }Q, \\
    \nabla\cdot u = 0 & \mbox{ in }Q, \\
    u = 0 & \mbox{ on }\Sigma, \\
    u(0) = 0 & \mbox{ in }\Om,
 \end{array}\right.
\end{equation*}
for some $p\in L^2(0,T;H^1(\Om))$, and there exists a constant $C>0$ depending only on $\Om$ such that
\begin{equation}\label{eq:regularity1}
\|u\|^2_{L^2(0,T;H^2(\Om)^N)} + \|u\|^2_{H^1(0,T;L^2(\Om)^N)}\leq C \| F\|^2_{L^2(Q)^N}.
\end{equation}
Furthermore, if $F\in L^2(0,T;H^2(\Om)^N)\cap H^1(0,T;L^2(\Om)^N)$, then\\ $u\in L^2(0,T;H^4(\Om)^N)\cap H^1(0,T;H^2(\Om)^N)$ and there exists a constant $C>0$ depending only on $\Om$ such that
\begin{equation}\label{eq:regularity2}
\begin{split}
\|u\|^2_{L^2(0,T;H^4(\Om)^N)} &+ \|u\|^2_{H^1(0,T;H^2(\Om)^N)}  \\ 
&\leq C\left( \| F\|^2_{L^2(0,T;H^2(\Om)^N)} + \| F\|^2_{H^1(0,T;L^2(\Om)^N)}\right).
\end{split}
\end{equation}
\end{lemma}

From now on, we set $N=3$, $i=2$ and $j=1$, i.e.,  we consider a control for the movement equation in \eqref{eq:Bouss} (and \eqref{eq:linearBouss}) of the form $v=(v_1,0,0)$. The arguments can be easily adapted to the general case by interchanging the roles of $i$ and $j$.

\subsection{Proof of Proposition \ref{prop:Carleman3}}

Let us introduce $(w,\pi_w)$, $(z,\pi_z)$ and $\widetilde\psi$, the solutions of the following systems:
 \begin{equation}\label{eq:w}
 \left\lbrace \begin{array}{ll}
     -w_t - \Delta w + \nabla \pi_w = \rho\, g & \mbox{ in }Q, \\
     \nabla\cdot w = 0 & \mbox{ in }Q, \\
     w = 0 & \mbox{ on }\Sigma, \\
     w(T) = 0 & \mbox{ in }\Om,
  \end{array}\right.
 \end{equation}
 \begin{equation}\label{eq:z}
 \left\lbrace \begin{array}{ll}
     -z_t - \Delta z + \nabla \pi_z = -\rho' \vp - \widetilde\psi \nabla \bar\theta & \mbox{ in }Q, \\
     \nabla\cdot z = 0 & \mbox{ in }Q, \\
     z = 0 & \mbox{ on }\Sigma, \\
     z(T) = 0 & \mbox{ in }\Om,
  \end{array}\right.
 \end{equation}
and
\begin{equation}\label{eq:tpsi}
 \left\lbrace \begin{array}{ll}
     -\tpsi_t - \Delta \tpsi = \rho\, g_0 + \rho\, \vp_3 - \rho' \psi & \mbox{ in }Q, \\
     \tpsi = 0 & \mbox{ on }\Sigma, \\
     \tpsi(T) = 0 & \mbox{ in }\Om,
  \end{array}\right.
 \end{equation}
 where $\rho(t)=e^{-\frac{3}{2}s\alpha^*}$. Adding \eqref{eq:w} and \eqref{eq:z}, we see that $(w+z,\pi_w+\pi_z,\tpsi)$ solves the same system as $(\rho\, \vp,\rho\, \pi,\rho\, \psi)$, where $(\vp,\pi,\psi)$ is the solution to \eqref{eq:adjointBouss}. By uniqueness of the Cauchy problem we have
 \begin{equation}\label{w+z}
 \rho\,\vp= w + z, \, \rho\,\pi = \pi_w + \pi_z \mbox{ and }\rho\, \psi = \tpsi.
 \end{equation}

Applying the divergence operator to \eqref{eq:z} we see that $\Delta \pi_z = -\nabla\cdot(\tpsi\nabla\bar\theta)$. We apply now the operator $\nabla\Delta=(\partial_1\Delta,\partial_2\Delta,\partial_3\Delta)$ to the equations satisfied by $z_1$ and $z_3$. We then have
\begin{equation}\label{eq:z1z3}
\begin{split}
-(\nabla\Delta z_1)_t - \Delta (\nabla\Delta z_1)= \nabla \left( \partial_1\nabla\cdot(\tpsi\nabla\bar\theta) - \Delta(\tpsi\partial_1\bar\theta) - \rho'\Delta\vp_1 \right)\mbox{ in }Q,\\
-(\nabla\Delta z_3)_t - \Delta (\nabla\Delta z_3)= \nabla \left( \partial_3\nabla\cdot(\tpsi\nabla\bar\theta) - \Delta(\tpsi\partial_3\bar\theta) - \rho'\Delta\vp_3 \right)\mbox{ in }Q.
\end{split}
\end{equation}

To the equations in \eqref{eq:z1z3}, we apply the Carleman inequality in Lemma \ref{lemma:Cnonhom} with $u=\nabla\Delta z_k$ for $k=1,3$ to obtain
\begin{multline}\label{11}
\sum_{k=1,3}\left[ \frac{1}{s}\iint\limits_Q e^{-2s\alpha}\frac{1}{\xi} |\nabla\nabla\Delta z_k|^2 \dxt
+ s\iint\limits_Q e^{-2s\alpha}\xi |\nabla\Delta z_k|^2 \dxt \right] \\
\leq C \left(  \sum_{k=1,3}\left[ s \int\limits_0^T\int\limits_{\om_0} e^{-2s\alpha}\xi |\nabla\Delta z_k|^2 \dxt  
 +  s^{-1/2} \left\|e^{-s\alpha^*}(\xi^*)^{-1/8} \nabla\Delta z_k \right\|^2_{L^2(\Sigma)^3} \right.\right. \\
\left.+   s^{-1/2} \left\|e^{-s\alpha^*}(\xi^*)^{-1/4} \nabla\Delta z_k \right\|^2_{H^{\frac{1}{4},\frac{1}{2}}(\Sigma)^3} 
+ \iint\limits_Q e^{-2s\alpha} |\rho'|^2 |\Delta \vp_k|^2 \dxt \right]\\
\left.+ \iint\limits_Q e^{-2s\alpha} (\sum_{k,l=1}^3|\partial^2_{kl} \tpsi|^2 + |\nabla \tpsi|^2 + |\tpsi|^2 ) \dxt   \right) ,
\end{multline}
for every $s\geq C$, where $C$ depends also on $\|\bar \theta\|_{L^{\infty}(0,T;W^{3,\infty}(\Om))}$.

Now, by Lemma \ref{lemma1} with $u=\Delta z_k$ for $k=1,3$ we have 
\begin{multline}\label{13}
\sum_{k=1,3} s^3\iint\limits_Q e^{-2s\alpha}\xi^3 |\Delta z_k|^2 \dxt  \\
\leq C \sum_{k=1,3} \left( s\iint\limits_Q e^{-2s\alpha}\xi |\nabla\Delta z_k|^2 \dxt
+ s^3\int\limits_0^T\int\limits_{\om_0} e^{-2s\alpha}\xi^3 |\Delta z_k|^2 \dxt\right),
\end{multline}
for every $s\geq C$, and by Lemma \ref{lemma2} with $u=z_k$ for $k=1,3$:
\begin{multline}\label{12}
\sum_{k=1,3}\left[  s^4\iint\limits_Q e^{-2s\alpha}\xi^4 |\nabla z_k|^2 \dxt 
+ s^6\iint\limits_Q e^{-2s\alpha}\xi^6 | z_k|^2 \dxt \right] \\
\leq C \sum_{k=1,3}\left[  s^3\iint\limits_Q e^{-2s\alpha}\xi^3 |\Delta z_k|^2 \dxt 
+ s^6\int\limits_0^T\int\limits_{\om_0} e^{-2s\alpha}\xi^6 | z_k|^2 \dxt \right],
\end{multline}
for every $s\geq C$.

Combining \eqref{11}, \eqref{13} and \eqref{12} and considering a nonempty open set $\om_1$ such that  $\om_0\Subset \om_1\Subset\om$ we obtain after some integration by parts

\begin{multline}\label{eq:14}
\sum_{k=1,3}\left[ \frac{1}{s}\iint\limits_Q e^{-2s\alpha}\frac{1}{\xi} |\nabla\nabla\Delta z_k|^2 \dxt
+ s\iint\limits_Q e^{-2s\alpha}\xi |\nabla\Delta z_k|^2 \dxt \right. \\
+s^3\iint\limits_Q e^{-2s\alpha}\xi^3 |\Delta z_k|^2 \dxt 
+ s^4\iint\limits_Q e^{-2s\alpha}\xi^4 |\nabla z_k|^2 \dxt 
\left.+ s^6\iint\limits_Q e^{-2s\alpha}\xi^6 | z_k|^2 \dxt \right] \\
\leq C \left(  \sum_{k=1,3}\left[ s^7 \int\limits_0^T\int\limits_{\om_1} e^{-2s\alpha}\xi^7 |z_k|^2 \dxt  
 +  s^{-1/2} \left\|e^{-s\alpha^*}(\xi^*)^{-1/8} \nabla\Delta z_k \right\|^2_{L^2(\Sigma)^3} \right.\right. \\
\left.+   s^{-1/2} \left\|e^{-s\alpha^*}(\xi^*)^{-1/4} \nabla\Delta z_k \right\|^2_{H^{\frac{1}{4},\frac{1}{2}}(\Sigma)^3} 
+ \iint\limits_Q e^{-2s\alpha} |\rho'|^2 |\Delta \vp_k|^2 \dxt \right]\\
\left.+ \iint\limits_Q e^{-2s\alpha} (\sum_{k,l=1}^3|\partial^2_{kl} \tpsi|^2 + |\nabla \tpsi|^2 +  |\tpsi|^2 ) \dxt   \right) ,
\end{multline}
for every $s\geq C$.

Notice that from the identities in \eqref{w+z}, the regularity estimate \eqref{eq:regularity1} for $w$ and $|\rho'|^2\leq Cs^2\rho^2(\xi)^{9/4}$ we obtain for $k=1,3$
\begin{equation*}
\begin{split}
\iint\limits_Q e^{-2s\alpha} |\rho'|^2 &|\Delta \vp_k|^2  \dxt 
= \iint\limits_Q e^{-2s\alpha} |\rho'|^2\rho^{-2} |\Delta (\rho\vp_k)|^2 \dxt \\
&\leq C s^2 \iint\limits_Q e^{-2s\alpha}\xi^{9/4} |\Delta z_k|^2 \dxt 
+ C s^2  \iint\limits_Q e^{-2s\alpha}\xi^{9/4} |\Delta w|^2 \dxt  \\
&\leq C s^2 \iint\limits_Q e^{-2s\alpha}\xi^{3} |\Delta z_k|^2 \dxt 
+ C \|\rho\,g\|^2_{L^2(Q)^3},
\end{split}
\end{equation*}
where we have also used the fact that $s^2e^{-2s\alpha}\xi^{9/4}$ is bounded and $1\leq C\xi^{3/4}$ in $Q$.

Now, from $z|_\Sigma=0$ and the divergence free condition we readily have (notice that $\alpha^*$ and $\xi^*$ do not depend on $x$)
\begin{equation*}
\begin{split}
s^4\iint\limits_Q  e^{-2s\alpha^*}(\xi^*)^4 |z_2|^2 \dxt 
&\leq C s^4\iint\limits_Q  e^{-2s\alpha^*}(\xi^*)^4 |\partial_2 z_2|^2 \dxt \\
&\leq C s^4\iint\limits_Q  e^{-2s\alpha}\xi^4 (|\nabla z_1|^2+|\nabla z_3|^2) \dxt.
\end{split}
\end{equation*}

Using these two last estimates in \eqref{eq:14}, we get
\begin{multline}\label{eq:Carlemannonh}
I(s,z):=\sum_{k=1,3}\left[ \frac{1}{s}\iint\limits_Q e^{-2s\alpha}\frac{1}{\xi} |\nabla\nabla\Delta z_k|^2 \dxt
+ s\iint\limits_Q e^{-2s\alpha}\xi |\nabla\Delta z_k|^2 \dxt \right. \\
+s^3\iint\limits_Q e^{-2s\alpha}\xi^3 |\Delta z_k|^2 \dxt 
+ s^4\iint\limits_Q e^{-2s\alpha}\xi^4 |\nabla z_k|^2 \dxt \\
\left.+ s^6\iint\limits_Q e^{-2s\alpha}\xi^6 | z_k|^2 \dxt \right] + s^4\iint\limits_Q  e^{-2s\alpha^*}(\xi^*)^4 |z_2|^2 \dxt \\
\leq C \left(   \sum_{k=1,3}\left[ s^7 \int\limits_0^T\int\limits_{\om_1} e^{-2s\alpha}\xi^7 |z_k|^2 \dxt  
 +  s^{-1/2} \left\|e^{-s\alpha^*}(\xi^*)^{-1/8} \nabla\Delta z_k \right\|^2_{L^2(\Sigma)^3} \right.\right. \\
\left.+   s^{-1/2} \left\|e^{-s\alpha^*}(\xi^*)^{-1/4} \nabla\Delta z_k \right\|^2_{H^{\frac{1}{4},\frac{1}{2}}(\Sigma)^3}  \right]
+  \|\rho\,g\|^2_{L^2(Q)^3} \\
\left.+ \iint\limits_Q e^{-2s\alpha} (\sum_{k,l=1}^3|\partial^2_{kl} \tpsi|^2 + |\nabla \tpsi|^2 + |\tpsi|^2 ) \dxt   \right) ,
\end{multline}
for every $s\geq C$.

For equation \eqref{eq:tpsi}, we use the classical Carleman inequality for the heat equation (see for example \cite{FurIma}): there exists $\widehat\lambda_3>0$ such that for any $\lambda >\widehat\lambda_3$ there exists \\ $C(\lambda,\Om,\om_1,\|\bar\theta\|_{L^{\infty}(0,T;W^{3,\infty}(\Om))})>0$ such that
\begin{multline}\label{eq:Carlemanheat}
J(s,\tpsi):=  s\iint\limits_Q e^{-2s\alpha}\xi (|\tpsi_t|^2 + \sum_{k,l=1}^3|\partial^2_{kl} \tpsi|^2 ) \dxt 
+ s^3\iint\limits_Q e^{-2s\alpha}\xi^3 |\nabla \tpsi|^2 \dxt \\
+ s^5\iint\limits_Q e^{-2s\alpha}\xi^5 |\tpsi|^2 \dxt 
\leq C\left(  s^2\iint\limits_Q e^{-2s\alpha}\xi^2 \rho^2(|g_0|^2+|\vp_3|^2) \dxt \right.  \\
\left.+ s^2\iint\limits_Q e^{-2s\alpha}\xi^2 |\rho'|^2|\rho|^{-2}|\tpsi|^2 \dxt 
+ s^5\int\limits_0^T\int\limits_{\om_1}e^{-2s\alpha}\xi^5 |\tpsi|^2 \dxt \right),
\end{multline}
for every $s\geq  C.$

We choose $\lambda_0$ in Proposition \ref{prop:Carleman3} (and Proposition \ref{prop:Carleman2}) to be $\lambda_0 := \max\{\widehat\lambda_0,\widehat\lambda_1,\widehat\lambda_2,\widehat\lambda_3\}$ and we fix $\lambda \geq \lambda_0$.

Combining inequalities \eqref{eq:Carlemannonh} and \eqref{eq:Carlemanheat}, and taking into account that $s^2e^{-2s\alpha}\xi^2 \rho^2$ is bounded, the identities in \eqref{w+z}, estimate \eqref{eq:regularity1} for $w$ and $|\rho'|\leq C s(\xi^*)^{9/8}\rho$ we have
\begin{multline}\label{eq:step0}
I(s,z) + J(s,\tpsi)
\leq C \left( \|\rho\, g\|^2_{L^2(Q)^3}
+ \|\rho\, g_0\|^2_{L^2(Q)}
+ s^5\int\limits_0^T\int\limits_{\om_1}e^{-2s\alpha}\xi^5 |\tpsi|^2 \dxt\right.\\
+ \sum_{k=1,3}\left[s^{-1/2} \left\|e^{-s\alpha^*}(\xi^*)^{-1/8} \nabla\Delta z_k \right\|^2_{L^2(\Sigma)^3} 
+   s^{-1/2} \left\|e^{-s\alpha^*}(\xi^*)^{-1/4} \nabla\Delta z_k \right\|^2_{H^{\frac{1}{4},\frac{1}{2}}(\Sigma)^3} \right. \\
\left.\left.+ s^7 \int\limits_0^T\int\limits_{\om_1} e^{-2s\alpha}\xi^7 |z_k|^2 \dxt \right]\right) , 
\end{multline}
for every $s\geq C$.

It remains to treat the boundary terms of this inequality and to eliminate the local term in $z_3$. 

\textbf{Estimate of the boundary terms.} First, we treat the first boundary term in \eqref{eq:step0}. Notice that, since $\alpha^*$ and $\xi^*$ do not depend on $x$, we can readily get by integration by parts, for $k=1,3$,
\begin{equation*}
\begin{split}
&\left\|e^{-s\alpha^*}\nabla\Delta z_k\right\|^2_{L^2(\Sigma)^3} \\
&\leq C \left\|s^{1/2}e^{-s\alpha^*}(\xi^*)^{1/2}\nabla\Delta z_k\right\|_{L^2(Q)^3}
\left\|s^{-1/2}e^{-s\alpha^*}(\xi^*)^{-1/2}\nabla \nabla\Delta z_k\right\|_{L^2(Q)^3} \\
& \leq C\left( s\iint\limits_Q e^{-2s\alpha^*}\xi^*|\nabla\Delta z_k|^2\dxt + \frac{1}{s}\iint\limits_Q e^{-2s\alpha^*}\frac{1}{\xi^*}|\nabla \nabla\Delta z_k|^2 \dxt \right),
\end{split}
\end{equation*}
so $\|e^{-s\alpha^*}\nabla\Delta z_k\|^2_{L^2(\Sigma)^3}$ is bounded by $I(s,z)$. On the other hand, we can bound the first boundary term as follows:
$$s^{-1/2}\left\|e^{-s\alpha^*}(\xi^*)^{-1/8}\nabla\Delta z_k\right\|^2_{L^2(\Sigma)^3}\leq C s^{-1/2} \left\|e^{-s\alpha^*}\nabla\Delta z_k\right\|^2_{L^2(\Sigma)^3}.$$
Therefore, the first boundary terms can be absorbed by taking $s$ large enough.

Now we treat the second boundary term in the right-hand side of \eqref{eq:step0}. We will use regularity estimates to prove that $z_1$ and $z_3$ multiplied by a certain weight function are regular enough. First, let us observe that from \eqref{w+z} and the regularity estimate \eqref{eq:regularity1} for $w$ we readily have
\begin{equation}\label{eq:1}
\|s^2e^{-s\alpha^*}(\xi^*)^2\rho\,\vp\|^2_{L^2(Q)^3} 
\leq C \left( I(s,z) + \left\|\rho\, g\right\|^2_{L^2(Q)^3}\right).
\end{equation}

We define now
$$\widetilde{z}:=se^{-s\alpha^*}(\xi^*)^{7/8}z,\,\widetilde\pi_z:=se^{-s\alpha^*}(\xi^*)^{7/8}\pi_z.$$
From \eqref{eq:z} we see that $(\widetilde{z},\widetilde\pi_z)$ is the solution of the Stokes system:
\begin{equation}\label{eq:ztilde}
\left\lbrace \begin{array}{ll}
    -\widetilde{z}_t - \Delta \widetilde{z} + \nabla \widetilde \pi_z = R_1 & \mbox{ in }Q, \\
    \nabla\cdot \widetilde{z} = 0 & \mbox{ in }Q, \\
    \widetilde{z} = 0 & \mbox{ on }\Sigma, \\
    \widetilde{z}(T) = 0 & \mbox{ in }\Om,
 \end{array}\right.
\end{equation}
where $R_1:=-se^{-s\alpha^*}(\xi^*)^{7/8}\rho' \vp - se^{-s\alpha^*}(\xi^*)^{7/8}\tpsi\nabla\bar\theta -  (se^{-s\alpha^*}(\xi^*)^{7/8})_t z.$ 
Taking into account that $|\alpha^*_t| \leq C (\xi^*)^{9/8}$, $|\rho'|\leq C s (\xi^*)^{9/8}\,\rho$, \eqref{regularity0} and \eqref{eq:1} we have
$$\| R_1 \|^2_{L^2(Q)^3}
\leq C\left( I(s,z) + J(s,\tpsi) + \|\rho g\|^2_{L^2(Q)^3} \right),$$
and therefore, by the regularity estimate \eqref{eq:regularity1} applied to \eqref{eq:ztilde}, we obtain
\begin{equation}\label{eq:2}
\|\widetilde z\|^2_{L^2(0,T;H^2(\Om)^3)\cap H^1(0,T;L^2(\Om)^3)} 
\leq C \left( I(s,z) + J(s,\tpsi) + \|\rho\, g\|^2_{L^2(Q)^3} \right).
\end{equation}

Next, let
$$\widehat{z}:=e^{-s\alpha^*}(\xi^*)^{-1/4}z,\,\widehat\pi_z:=e^{-s\alpha^*}(\xi^*)^{-1/4}\pi_z.$$
From \eqref{eq:z}, $(\widehat{z},\widehat\pi_z)$ is the solution of the Stokes system:
\begin{equation}\label{eq:zhat}
\left\lbrace \begin{array}{ll}
    -\widehat{z}_t - \Delta \widehat{z} + \nabla \widehat\pi_z = R_2 & \mbox{ in }Q, \\
    \nabla\cdot \widehat{z} = 0 & \mbox{ in }Q, \\
    \widehat{z} = 0 & \mbox{ on }\Sigma, \\
    \widehat{z}(T) = 0 & \mbox{ in }\Om,
 \end{array}\right.
\end{equation}
where $R_2:=-e^{-s\alpha^*}(\xi^*)^{-1/4}\rho' \vp - e^{-s\alpha^*}(\xi^*)^{-1/4}\tpsi\nabla\bar\theta -  (e^{-s\alpha^*}(\xi^*)^{-1/4})_t z$. By the same arguments as before, and thanks to \eqref{eq:2}, we can easily prove that $R_2\in L^2(0,T;H^2(\Om)^3)\cap H^1(0,T;L^2(\Om)^3)$ (for the first term in $R_2$, we use again \eqref{w+z} and \eqref{eq:2}) and furthermore
$$\| R_2 \|^2_{L^2(0,T;H^2(\Om)^3)\cap H^1(0,T;L^2(\Om)^3)}
\leq C \left( I(s,z) + J(s,\tpsi) + \|\rho\, g\|^2_{L^2(Q)^3} \right).$$
By the regularity estimate \eqref{eq:regularity2} applied to \eqref{eq:zhat}, we have
\begin{equation*}
\|\widehat z\|^2_{L^2(0,T;H^4(\Om)^3)\cap H^1(0,T;H^2(\Om)^3)} 
\leq C \left( I(s,z) + J(s,\tpsi) + \|\rho\, g\|^2_{L^2(Q)^3} \right).
\end{equation*}
In particular, $e^{-s\alpha^*}(\xi^*)^{-1/4}\nabla\Delta z_k \in L^2(0,T;H^1(\Om)^3)\cap H^1(0,T;H^{-1}(\Om)^3)$ for $k=1,3$ and
\begin{multline}\label{eq:3}
\sum_{k=1,3}\| e^{-s\alpha^*}(\xi^*)^{-1/4}\nabla\Delta z_k \|^2_{L^2(0,T;H^1(\Om)^3)} 
+ \| e^{-s\alpha^*}(\xi^*)^{-1/4}\nabla\Delta z_k \|^2_{H^1(0,T;H^{-1}(\Om)^3)} \\
\leq C \left( I(s,z) + J(s,\tpsi) + \|\rho\, g\|^2_{L^2(Q)^3} \right).
\end{multline}

To end this part, we use a trace inequality to estimate the second boundary term  in the right-hand side of \eqref{eq:step0}:
\begin{equation*}
\begin{split}
\sum_{k=1,3} s^{-1/2} & \left\|e^{-s\alpha^*}(\xi^*)^{-1/4}\nabla\Delta z_k\right\|^2_{H^{\frac{1}{4},\frac{1}{2}}(\Sigma)^3} \\ 
& \leq  C\,s^{-1/2} \sum_{k=1,3} \left[ \left\| e^{-s\alpha^*}(\xi^*)^{-1/4}\nabla\Delta z_k \right\|^2_{L^2(0,T;H^1(\Om)^3)} \right. \\
&\left. \qquad\qquad\qquad+ \left\| e^{-s\alpha^*}(\xi^*)^{-1/4}\nabla\Delta z_k \right\|^2_{H^1(0,T;H^{-1}(\Om)^3)} \right],
\end{split}
\end{equation*}
By taking $s$ large enough in \eqref{eq:step0}, the boundary terms\\ $s^{-1/2}\|e^{-s\alpha}\xi^{-1/4}\nabla\Delta z_k\|^2_{H^{\frac{1}{4},\frac{1}{2}}(\Sigma)^3}$ can be absorbed by the terms in the left-hand side of \eqref{eq:3}.

Thus, using  \eqref{w+z} and \eqref{eq:regularity1} for $w$ in the right-hand side of \eqref{eq:step0}, we have for the moment
\begin{multline*}
I(s,z) + J(s,\tpsi)
\leq C \left( \|\rho\, g\|^2_{L^2(Q)^3}
+ \|\rho\, g_0\|^2_{L^2(Q)}
+ s^5\int\limits_0^T\int\limits_{\om_1}e^{-2s\alpha}\xi^5 |\tpsi|^2 \dxt\right.\\
\left.+ s^7 \int\limits_0^T\int\limits_{\om_1} e^{-2s\alpha}\xi^7 \rho^2|\vp_1|^2 \dxt+ 
 s^7 \int\limits_0^T\int\limits_{\om_1} e^{-2s\alpha}\xi^7 \rho^2|\vp_3|^2 \dxt\right), 
\end{multline*}
for every $s\geq C$.
Furthermore, notice that using again \eqref{w+z}, \eqref{eq:regularity1} for $w$ and \eqref{eq:2} we obtain from the previous inequality
\begin{multline}\label{eq:step1}
s^2\iint\limits_Q e^{-2s\alpha^*}(\xi^*)^{7/4}\rho^2 |\vp_{3,t}|^2 \dxt +  \widetilde I(s,\rho\, \vp) + J(s,\tpsi) \\
\leq C \left( \|\rho\, g\|^2_{L^2(Q)^3}
+ \|\rho\, g_0\|^2_{L^2(Q)^3}
+ s^5\int\limits_0^T\int\limits_{\om_1}e^{-2s\alpha}\xi^5 |\tpsi|^2 \dxt\right.\\
\left.+ s^7 \int\limits_0^T\int\limits_{\om_1} e^{-2s\alpha}\xi^7 \rho^2|\vp_1|^2 \dxt
+ s^7 \int\limits_0^T\int\limits_{\om_1} e^{-2s\alpha}\xi^7 \rho^2|\vp_3|^2 \dxt \right) , 
\end{multline}
for every $s\geq C$, where
\begin{multline*}
\widetilde I(s,\rho\,\vp):= 
\sum_{k=1,3}\left[ 
s^3\iint\limits_Q e^{-2s\alpha}\xi^3 \rho^2|\Delta \vp_k|^2 \dxt 
+ s^4\iint\limits_Q e^{-2s\alpha}\xi^4 \rho^2|\nabla \vp_k|^2 \dxt \right.\\
\left.+ s^6\iint\limits_Q e^{-2s\alpha}\xi^6 \rho^2| \vp_k|^2 \dxt \right] + s^4\iint\limits_Q  e^{-2s\alpha^*}(\xi^*)^4 \rho^2|\vp_2|^2 \dxt.   
\end{multline*}

\textbf{Estimate of $\vp_3$.} We deal in this part with the last term in the right-hand side of \eqref{eq:step1}. We introduce a function $\zeta_1\in C^2_0(\om)$ such that $\zeta_1\geq 0$ and $\zeta_1=1$ in $\om_1$, and using equation \eqref{eq:tpsi} we have
\begin{equation*}
\begin{split}
C\,s^7 \int\limits_0^T\int\limits_{\om_1} e^{-2s\alpha}\xi^7 \rho^2 &|\vp_3|^2 \dxt 
\leq C s^7 \int\limits_0^T\int\limits_{\om}\zeta_1 e^{-2s\alpha}\xi^7 \rho^2|\vp_3|^2 \dxt \\
& = C s^7 \int\limits_0^T\int\limits_{\om}\zeta_1 e^{-2s\alpha}\xi^7 \rho\,\vp_3 ( -\tpsi_t - \Delta \tpsi - \rho\, g_0 + \rho'\psi ) \dxt,
\end{split}
\end{equation*}
and we integrate by parts in this last term, in order to estimate it by local integrals of $\tpsi$, $g_0$ and $\epsilon\, I(s,\rho\, \vp)$. This approach was already introduced in \cite{E&S&O&P-N-1}.

We first integrate by parts in time taking into account that\\ $e^{-2s\alpha(0)}\xi^7(0)=e^{-2s\alpha(T)}\xi^7(T)=0$:
\begin{equation*}
\begin{split}
-&C s^7 \int\limits_0^T\int\limits_{\om}\zeta_1 e^{-2s\alpha}\xi^7 \rho\,\vp_3 \tpsi_t  \dxt \\  
&=C s^7 \int\limits_0^T\int\limits_{\om}\zeta_1 e^{-2s\alpha}\xi^7 \rho\,\vp_{3,t} \tpsi  \dxt 
+ C s^7 \int\limits_0^T\int\limits_{\om}\zeta_1 (e^{-2s\alpha}\xi^7 \rho)_t\,\vp_3 \tpsi  \dxt \\
& \leq \epsilon \left( s^2 \int\limits_0^T\int\limits_{\om} e^{-2s\alpha^*}(\xi^*)^{7/4} \rho^2|\vp_{3,t}|^2  \dxt + \widetilde I(s,\rho\,\vp) \right) \\
&\qquad +C(\lambda,\epsilon) \,\left( s^{12} \int\limits_0^T\int\limits_{\om} e^{-4s\alpha+2s\alpha^*}\xi^{49/4} |\tpsi|^2  \dxt 
+ s^{10} \int\limits_0^T\int\limits_{\om} e^{-2s\alpha}\xi^{41/4} |\tpsi|^2  \dxt \right) ,
\end{split}
\end{equation*}
where we have used that $$|(e^{-2s\alpha}\xi^7\rho)_t|\leq C se^{-2s\alpha}\xi^{65/8}\rho$$ and Young's inequality. Now we integrate by parts in space:
\begin{equation*}
\begin{split}
-&C s^7 \int\limits_0^T\int\limits_{\om}\zeta_1 e^{-2s\alpha}\xi^7 \rho\,\vp_3 \Delta\tpsi  \dxt  
= -C s^7 \int\limits_0^T\int\limits_{\om}\zeta_1 e^{-2s\alpha}\xi^7 \rho\,\Delta\vp_3 \tpsi  \dxt \\ 
& -2C s^7 \int\limits_0^T\int\limits_{\om} \nabla(\zeta_1 e^{-2s\alpha}\xi^7)\cdot \rho\,\nabla\vp_3 \tpsi  \dxt 
- C\,s^7 \int\limits_0^T\int\limits_{\om} \Delta(\zeta_1 e^{-2s\alpha}\xi^7) \rho\,\vp_3 \tpsi  \dxt \\
& \leq \epsilon\, \widetilde I(s,\rho\,\vp) + C(\epsilon)\, s^{12} \int\limits_0^T\int\limits_{\om} e^{-2s\alpha}\xi^{12} |\tpsi|^2  \dxt  ,
\end{split}
\end{equation*}
where we have used that 
$$\nabla(\zeta_1e^{-2s\alpha}\xi^7)\leq C se^{-2s\alpha}\xi^8\mbox{ and }\Delta(\zeta_1e^{-2s\alpha}\xi^7)\leq C s^2e^{-2s\alpha}\xi^9,$$
and Young's inequality.

Finally,
\begin{equation*}
\begin{split}
&C s^7 \int\limits_0^T\int\limits_{\om}\zeta_1 e^{-2s\alpha}\xi^7 \rho\,\vp_3 (-\rho\, g_0 + \rho'\, \psi)  \dxt \\
&\leq C s^7 \int\limits_0^T\int\limits_{\om}\zeta_1 e^{-2s\alpha}\xi^7 \rho\,|\vp_3| (\rho|g_0| + Cs\xi^{9/8}\, |\tpsi|)  \dxt \\
& \leq\epsilon\, \widetilde I(s,\rho\,\vp) 
+ C(\epsilon)\left( s^8 \int\limits_0^T\int\limits_{\om} e^{-2s\alpha}\xi^8 \rho^2\, |g_0|^2   \dxt 
+ s^{10} \int\limits_0^T\int\limits_{\om} e^{-2s\alpha}\xi^{41/4} |\tpsi|^2   \dxt \right).
\end{split}
\end{equation*}

Setting $\epsilon=1/2$ and noticing that
$$e^{-2s\alpha}\leq e^{-4s\alpha+2s\alpha^*}\mbox{ in }Q,$$
(see \eqref{pesos}) we obtain \eqref{eq:Carleman3} from \eqref{eq:step1}. This completes the proof of Proposition \ref{prop:Carleman3}.

\section{Null controllability of the linear system}

Here we are concerned with the null controllability of the system
\begin{equation}\label{eq:linearBouss2}
\left\lbrace \begin{array}{ll}
    L y  + \nabla p = f + (v_1,0,0)\1_{\om} + \theta\,e_3& \mbox{ in }Q, \\
    L \theta + y\cdot \nabla\bar\theta = f_0 + v_0\1_{\om} & \mbox{ in }Q, \\
    \nabla\cdot y = 0 & \mbox{ in }Q, \\
    y = 0,\, \theta=0 & \mbox{ on }\Sigma, \\
    y(0) = y^0,\, \theta(0)=\theta^0 & \mbox{ in }\Om,
 \end{array}\right.
\end{equation}
where $y^0\in V$, $\theta^0\in H^1_0(\Om)$, $f$ and $f_0$ are in appropriate weighted spaces, the controls $v_0$ and $ v_1$ are in $L^2(\om\times (0,T))$ and 
$$L q= q_t - \Delta q .$$

Before dealing with the null controllability of \eqref{eq:linearBouss2}, we will deduce a Carleman inequality with weights not vanishing at $t=0$. To this end, let us introduce the following weight functions:
\begin{equation}\label{pesos2}
\begin{split}
&\beta(x,t) = \dfrac{e^{2\lambda\|\eta\|_{\infty}}-e^{\lambda\eta(x)}}{\tell^{8}(t)},\, \gamma(x,t)=\dfrac{e^{\lambda\eta(x)}}{\tell^8(t)},\\
&\beta^*(t) = \max_{x\in\overline{\Om}} \beta(x,t),\, \gamma^*(t) = \min_{x\in\overline{\Om}} \gamma(x,t),\\
&\widehat{\beta}(t) = \min_{x\in\overline{\Om}} \beta(x,t),\, \widehat{\gamma}(t) = \max_{x\in\overline{\Om}} \gamma(x,t),
\end{split}
\end{equation}
where
\begin{equation*}
\tell(t)= \left\lbrace
\begin{array}{ll}
\|\ell\|_{\infty} & 0\leq t \leq T/2, \\
\ell(t) & T/2< t \leq T. 
\end{array}
\right.
\end{equation*}
 
\begin{lemma}\label{lemma:newCarleman3}
Assume $N=3$. Let $s$ and $\lambda$ be like in Proposition \ref{prop:Carleman3} and $(\bar p,\bar\theta)$ satisfy \eqref{eq:ncBouss0}-\eqref{regularity0}. Then, there exists a constant $C>0$ (depending on $s$, $\lambda$ and $\bar\theta$) such that every solution $(\vp,\pi,\psi)$ of \eqref{eq:adjointBouss} satisfies:
\begin{multline}\label{eq:newCarleman3}
\iint\limits_{Q}  e^{-5s\beta^*}  (\gamma^*)^4 |\vp|^2 dx\,dt + \iint\limits_{Q}  e^{-5s\beta^*}  (\gamma^*)^5 |\psi|^2 dx\,dt + \|\vp(0)\|^2_{L^2(\Om)^3} + \|\psi(0)\|^2_{L^2(\Om)} \\
\leq  C \left( \iint\limits_{Q}  e^{-3s \beta^*}(|g|^2 +|g_0|^2 ) dx\,dt 
+   \int\limits_0^T\int\limits_{\om}e^{-2s\widehat\beta - 3s\beta^*}\widehat\gamma^7|\vp_1|^2 dx\,dt \right. \\
\left. +    \int\limits_0^T\int\limits_{\om}e^{-4s\widehat\beta - s\beta^*}\widehat\gamma^{49/4}|\psi|^2 dx\,dt \right).
\end{multline}
\end{lemma}

Let us also state this result for $N=2$.
\begin{lemma}\label{lemma:newCarleman2}
Assume $N=2$. Let $s$ and $\lambda$ be like in Proposition \ref{prop:Carleman2} and $(\bar p,\bar\theta)$ satisfy \eqref{eq:ncBouss0}-\eqref{regularity0}. Then, there exists a constant $C>0$ (depending on $s$, $\lambda$ and $\bar\theta$) such that every solution $(\vp,\pi,\psi)$ of \eqref{eq:adjointBouss} satisfies:
\begin{multline}\label{eq:newCarleman2}
\iint\limits_{Q}  e^{-5s\beta^*}  (\gamma^*)^4 |\vp|^2 dx\,dt + \iint\limits_{Q}  e^{-5s\beta^*}  (\gamma^*)^5 |\psi|^2 dx\,dt + \|\vp(0)\|^2_{L^2(\Om)^2} + \|\psi(0)\|^2_{L^2(\Om)} \\
\leq  C \left( \iint\limits_{Q}  e^{-3s \beta^*}(|g|^2 +|g_0|^2 ) dx\,dt 
+   \int\limits_0^T\int\limits_{\om}e^{-4s\widehat\beta - s\beta^*}\widehat\gamma^{49/4}|\psi|^2 dx\,dt \right).
\end{multline}
\end{lemma}

\begin{proof}[Proof of Lemma \ref{lemma:newCarleman3}:] We start by an a priori estimate for system \eqref{eq:adjointBouss}. To do this, we introduce a function $\nu\in C^1([0,T])$ such that
$$\nu \equiv 1 \mbox{ in }[0,T/2],\, \nu \equiv 0 \mbox{ in } [3T/4,T].$$
We easily see that $(\nu\vp,\nu\pi,\nu\psi)$ satisfies
\begin{equation*}
\left\lbrace \begin{array}{ll}
    -(\nu\vp)_t - \Delta (\nu\vp)  + \nabla (\nu\pi) = \nu\,g - (\nu\psi)\nabla\bar\theta - \nu'\vp& \mbox{ in }Q, \\
    -(\nu\psi)_t - \Delta (\nu\psi) = \nu\,g_0 + \nu\,\vp_3 - \nu'\psi & \mbox{ in }Q, \\
    \nabla\cdot (\nu\vp) = 0 & \mbox{ in }Q, \\
    (\nu\vp) = 0,\, (\nu\psi)=0 & \mbox{ on }\Sigma, \\
    (\nu\vp)(T) = 0,\, (\nu\psi)(T)= 0 & \mbox{ in }\Om,
 \end{array}\right.
\end{equation*}
thus we have the energy estimate
\begin{multline*}
\|\nu\vp\|^2_{L^2(0,T;V)} + \|\nu\vp\|^2_{L^{\infty}(0,T;H)} 
+\|\nu\psi\|^2_{L^2(0,T;H^1(\Om))} + \|\nu\psi\|^2_{L^{\infty}(0,T;L^2(\Om)
)} \\
\leq C( \|\nu g\|^2_{L^2(Q)^3} + \|\nu' \vp\|^2_{L^2(Q)^3} + \|\nu g_0\|^2_{L^2(Q)} + \|\nu' \psi\|^2_{L^2(Q)}).
\end{multline*}
Using the properties of the function $\nu$, we readily obtain
\begin{equation*}
\begin{split}
\|\vp  \|^2_{L^2(0,T/2;H)}  &+ \|\vp(0)\|^2_{L^2(\Om)^3} 
+ \|\psi  \|^2_{L^2(0,T/2;L^2(\Om))}  + \|\psi(0)\|^2_{L^2(\Om)}\\
&\leq C\left(  \| g\|^2_{L^2(0,3T/4;L^2(\Om)^3)} + \|\vp\|^2_{L^2(T/2,3T/4;L^2(\Om)^3)}\right.\\
&\qquad \left.+\| g_0\|^2_{L^2(0,3T/4;L^2(\Om))} + \|\psi\|^2_{L^2(T/2,3T/4;L^2(\Om))} \right).
\end{split}
\end{equation*}
From this last inequality, and the fact that
$$e^{-3s\beta^*}\geq C>0, \,\forall t\in [0,3T/4] \mbox{ and }e^{-5s\alpha^*}(\xi^*)^4\geq C >0,\, \forall t\in [T/2,3T/4] $$
we have
\begin{multline}\label{eq:4}
\int\limits_0^{T/2}\int\limits_{\Om} e^{-5s\beta^*}(\gamma^*)^4 |\vp|^2 \dxt 
+\int\limits_0^{T/2}\int\limits_{\Om} e^{-5s\beta^*}(\gamma^*)^5 |\psi|^2 \dxt \\
+\|\vp(0)\|^2_{L^2(\Om)^3}
+\|\psi(0)\|^2_{L^2(\Om)}
\leq C\left( \int\limits_0^{3T/4}\int\limits_{\Om} e^{-3s\beta^*}(|g|^2+|g_0|^2) \dxt \right.\\
\left.+ \int\limits_{T/2}^{3T/4}\int\limits_{\Om} e^{-5s\alpha^*}(\xi^*)^4|\vp|^2 \dxt
+ \int\limits_{T/2}^{3T/4}\int\limits_{\Om} e^{-5s\alpha^*}(\xi^*)^5|\psi|^2 \dxt \right).
\end{multline}
Note that the last two terms in \eqref{eq:4} are bounded by the left-hand side of the Carleman inequality \eqref{eq:Carleman3}. Since $\alpha=\beta$ in $\Om\times (T/2,T)$, we have:
\begin{equation*}
\begin{split}
\int\limits_{T/2}^T\int\limits_{\Om} e^{-5s\beta^*}(\gamma^*)^4 |\vp|^2 \dxt 
+\int\limits_{T/2}^T\int\limits_{\Om} e^{-5s\beta^*}(\gamma^*)^5 |\psi|^2 \dxt\\
= \int\limits_{T/2}^T\int\limits_{\Om} e^{-5s\alpha^*}(\xi^*)^4 |\vp|^2 \dxt
+ \int\limits_{T/2}^T\int\limits_{\Om} e^{-5s\alpha^*}(\xi^*)^5 |\psi|^2 \dxt \\ 
\leq  \iint\limits_Q e^{-5s\alpha^*}(\xi^*)^4|\vp|^2\dxt
+ \iint\limits_Q e^{-5s\alpha^*}(\xi^*)^5 |\psi|^2\dxt .
\end{split}
\end{equation*}
Combining this with the Carleman inequality \eqref{eq:Carleman3}, we deduce
\begin{multline*}
\int\limits_{T/2}^{T}\int\limits_{\Om} e^{-5s\beta^*}(\gamma^*)^4 |\vp|^2 \dxt 
 +\int\limits_{T/2}^T\int\limits_{\Om} e^{-5s\beta^*}(\gamma^*)^5 |\psi|^2 \dxt \\
\leq C\left( \iint\limits_Q e^{-3s\alpha^*}(|g|^2+|g_0|^2) \dxt
+\int\limits_{0}^{T}\int\limits_{\om} e^{-2s\widehat{\alpha}-3s\alpha^*}(\widehat \xi)^7|\vp_1|^2 \dxt \right. \\
\left. + \int\limits_{0}^{T}\int\limits_{\om} e^{-4s\widehat{\alpha}-s\alpha^*}(\widehat \xi)^{49/4}|\psi|^2 \dxt \right).
\end{multline*}
Since 
$$e^{-3s\beta^*}, e^{-2s\widehat{\beta}-3s\beta^*}\widehat{\gamma}^7, e^{-2s\widehat{\beta}-3s\beta^*}\widehat{\gamma}^7, e^{-4s\widehat{\beta}-s\beta^*}\widehat{\gamma}^{49/4} \geq C>0, \,\forall t\in [0,T/2],$$
we can readily get
\begin{multline*}
\int\limits_{T/2}^{T}\int\limits_{\Om} e^{-5s\beta^*}(\gamma^*)^4 |\vp|^2 \dxt
+ \int\limits_{T/2}^{T}\int\limits_{\Om} e^{-5s\beta^*}(\gamma^*)^5 |\psi|^2 \dxt \\
\leq C\left( \iint\limits_Q e^{-3s\beta^*}(|g|^2+|g_0|^2) \dxt 
+ \int\limits_{0}^{T}\int\limits_{\om} e^{-2s\widehat{\beta}-3s\beta^*} \widehat{\gamma}^7 |\vp_1|^2 \dxt \right.\\
+ \left.  \int\limits_{0}^{T}\int\limits_{\om} e^{-4s\widehat{\beta}-s\beta^*} \widehat{\gamma}^{49/4} |\psi|^2 \dxt \right),
\end{multline*}
which, together with \eqref{eq:4}, yields \eqref{eq:newCarleman3}.
\end{proof}

Now we will prove the null controllability of \eqref{eq:linearBouss2}. Actually, we will prove the existence of a solution for this problem in an appropriate weighted space. Let us introduce the space
\begin{equation*}
\begin{array}{l}
E=\{\, (y,p,v_1,\theta,v_0):
 e^{3/2s\beta^*}\,y,\,e^{s\widehat{\beta}+3/2s\beta^*}\widehat \gamma^{-7/2}\,(v_1,0,0)\1_{\om}\in L^2(Q)^3, \\
\noalign{\medskip}\hskip1cm e^{3/2s\beta^*}\,\theta,\,e^{2s\widehat{\beta}+1/2s\beta^*}\widehat \gamma^{-49/8}\,v_0\1_{\om} \in L^2(Q) ,  \\
 \noalign{\medskip}\hskip1cm e^{3/2 s\beta^*}(\gamma^*)^{-9/8}y \in L^2(0,T;H^2(\Om)^3)\cap L^{\infty}(0,T;V),\\
\noalign{\medskip}\hskip1cm e^{3/2 s\beta^*}(\gamma^*)^{-9/8}\theta \in L^2(0,T;H^2(\Om))\cap L^{\infty}(0,T;H^1_0(\Om)),\\
\noalign{\medskip}\hskip1cm \, e^{5/2s\beta^*}(\gamma^*)^{-2}(Ly+\nabla
p - \theta e_3 -(v_1,0,0)\1_{\om}) \in L^2(Q)^3,\\
\noalign{\medskip}\hskip1cm \, e^{5/2s\beta^*}(\gamma^*)^{-5/2}(L\theta+y\cdot \nabla \bar\theta-v_0\1_{\om}) \in L^2(Q)\,\}.
\end{array}
\end{equation*}

It is clear that $E$ is a Banach space for the following norm:
$$
\begin{array}{l}
\displaystyle \|(y,p,v_1,\theta,v_0)\|_{E}=\left( \|e^{3/2s\beta^*}\,y\|^2_{L^2(Q)^3}
+\|e^{s\widehat{\beta}+3/2s\beta^*}\widehat\gamma^{-7/2}\,v_1\1_{\om}\|^2_{L^2(Q)}\right.\\
\qquad\qquad +\|e^{3/2s\beta^*}\,\theta\|^2_{L^2(Q)}
+\|e^{2s\widehat{\beta}+1/2s\beta^*}\widehat\gamma^{-49/8}\,v_0\1_{\om}\|^2_{L^2(Q)}\\
\qquad\qquad + \|e^{3/2 s\beta^*}(\gamma^*)^{-9/8}\,y\|^2_{L^2(0,T;H^2(\Om)^3)} 
\displaystyle +\|e^{3/2 s\beta^*}(\gamma^*)^{-9/8}y\|^2_{L^\infty(0,T;V)}\\
\qquad\qquad + \|e^{3/2 s\beta^*}(\gamma^*)^{-9/8}\,\theta\|^2_{L^2(0,T;H^2(\Om))} 
\displaystyle +\|e^{3/2 s\beta^*}(\gamma^*)^{-9/8}\theta\|^2_{L^\infty(0,T;H^1_0)}\\
\qquad\qquad +\|e^{5/2s\beta^*}(\gamma^*)^{-2}(Ly+\nabla
p - \theta e_3 -(v_1,0,0)\1_{\om})\|^2_{L^2(Q)^3}\\ 
\qquad\qquad\left. +\|e^{5/2s\beta^*}(\gamma^*)^{-5/2}(L\theta+y\cdot \nabla \bar\theta-v_0\1_{\om})\|^2_{L^2(Q)} \right)^{1/2} 
\end{array}
$$

\begin{remark}
Observe in particular that $(y,p,v_1,\theta,v_0)\in E$ implies $y(T)=0$ and $\theta(T)=0$ in $\Om$. Moreover, the functions belonging to this space posses the interesting following property:
$$e^{5/2s\beta^*}(\gamma^*)^{-2}(y\cdot \nabla)y\in L^2(Q)^3\mbox{ and }
e^{5/2s\beta^*}(\gamma^*)^{-5/2}y\cdot \nabla\theta\in L^2(Q).$$
\end{remark}

\begin{proposition}\label{prop:null}
Assume $N=3$, $(\bar p,\bar\theta)$ satisfies \eqref{eq:ncBouss0}-\eqref{regularity0} and 
\begin{equation*}
y^0\in V,\,\theta_0 \in H^1_0(\Om),\, e^{5/2s\beta^*}(\gamma^*)^{-2}f\in L^2(Q)^3 \mbox{ and } e^{5/2s\beta^*}(\gamma^*)^{-5/2}f_0\in L^2(Q).
\end{equation*}
Then, we can find controls $v_1$ and $v_0$ such that the associated solution $(y,p,\theta)$ to \eqref{eq:linearBouss2} satisfies \\
$(y,p,v_1,\theta,v_0)\in E$. In particular, $y(T)=0$ and $\theta(T)=0$.
\end{proposition}

\begin{proof}[ Sketch of the proof:]
The proof of this proposition is very similar to the one of Proposition 2
in \cite{Gue} (see also Proposition 2 in \cite{E&S&O&P} and Proposition 3.3 in \cite{CarrGue}), so we will just give the main ideas.

Following the arguments in \cite{FurIma} and \cite{OlegN-S}, we introduce the space
$$P_0=\{\,(\chi,\sigma,\kappa)\in C^2(\overline Q)^{5}:\nabla\cdot \chi=0,\, \chi=0 \mbox{ on }
\Sigma,\, \kappa = 0 \mbox{ on }\Sigma \}$$
and we consider the following variational problem: find $(\widehat\chi,\widehat\sigma,\widehat\kappa)\in P_0$ such that
\begin{equation}\label{48p}
a((\widehat \chi,\widehat \sigma,\widehat\kappa),(\chi,\sigma,\kappa))=\langle G,(\chi,\sigma,\kappa)\rangle \quad
\forall (\chi,\sigma,\kappa) \in P_0,
\end{equation}
where we have used the notations
$$\begin{array}{l}
\displaystyle a((\widehat \chi,\widehat \sigma,\widehat\kappa),(\chi,\sigma,\kappa))
=\iint\limits_Q
e^{-3s\beta^*}\,(L^*\widehat \chi+\nabla \widehat \sigma + \widehat\kappa\nabla\bar\theta)\cdot(L^*\chi+\nabla \sigma + \kappa\nabla\bar\theta)\,\dxt\\ 
\displaystyle + \iint\limits_Q e^{-3s\beta^*}\,(L^*\widehat \kappa-\widehat\chi_3)(L^*\kappa-\chi_3)\,\dxt
+ \int\limits_0^T\int\limits_{\om} e^{-2s\widehat{\beta}-3s\beta^*}\widehat \gamma^7\,\widehat \chi_1\,\chi_1\,\dxt \\ 
\noalign{\medskip}\displaystyle
\hskip 5cm + \int\limits_0^T\int\limits_{\om} e^{-4s\widehat{\beta}-s\beta^*}\widehat \gamma^{49/4}\,\widehat \kappa\,\kappa\,\dxt,
\end{array}$$
$$\langle G,(\chi,\sigma,\kappa)\rangle =\iint\limits_{Q}  f\cdot \chi\,\dxt 
+ \iint\limits_{Q}  f_0\, \kappa
\,\dxt + \int\limits_{\Omega}y^0\cdot \chi(0)\,dx 
+ \int\limits_{\Omega}\theta^0\, \kappa(0)\,dx$$
and $L^*$ is the adjoint operator of $L$, i.e.
   $$
L^*q = -q_t - \Delta q.
   $$

   It is clear that $a(\cdot\,,\cdot\,,\cdot):P_0\times P_0\mapsto\R$ is a symmetric,
definite positive bilinear form on $P_0$. We denote by $P$ the completion of $P_0$ for the norm induced by $a(\cdot\,,\cdot\,,\cdot)$. Then $a(\cdot\,,\cdot\,,\cdot)$ is well-defined, continuous and again definite positive on $P$. Furthermore, in view of the Carleman estimate \eqref{eq:newCarleman3}, the linear form $(\chi,\sigma,\kappa) \mapsto \langle G,(\chi,\sigma,\kappa)\rangle$ is well-defined and continuous on $P$.
   Hence, from Lax-Milgram's lemma, we deduce that the
variational problem
   \begin{equation*}
   \left\{
   \begin{array}{l}
   \displaystyle a((\widehat \chi,\widehat \sigma,\widehat\kappa),(\chi,\sigma,\kappa))=\langle
   G,(\chi,\sigma,\kappa)\rangle
   \\ \noalign{\medskip}\displaystyle
   \forall (\chi,\sigma,\kappa) \in P, \quad (\widehat \chi,\widehat \sigma,\widehat\kappa) \in P,
   \end{array}
   \right.
   \end{equation*}
   possesses exactly one solution $(\widehat \chi,\widehat \sigma,\widehat\kappa)$.

Let $\widehat y$, $\widehat v_1$, $\widehat\theta $ and $\widehat v_0 $ be given by 
\begin{equation*}
\left\{\begin{array}{ll} \displaystyle \widehat
y=e^{-3s\beta^*}(L^*\widehat \chi+\nabla \widehat \sigma + \widehat\kappa \nabla \bar\theta),&\mbox{ in }Q,
\\ \noalign{\smallskip}\displaystyle
\widehat v_1=-e^{-2s\widehat{\beta}-3s\beta^*}\widehat\gamma^7\,\widehat \chi_1,&\mbox{ in }\om\times (0,T),\\
\widehat\theta = e^{-3s\beta^*}(L^* \widehat\kappa - \widehat\chi_3),&\mbox{ in }Q,\\
\widehat v_0 = - e^{-4s\widehat{\beta}-s\beta^*}\widehat\gamma^{49/4}\,\widehat \kappa,&\mbox{ in }\om\times (0,T).
\end{array}\right.
\end{equation*}
Then, it is readily seen that they satisfy
\begin{multline*}
\iint\limits_{Q}e^{3s\beta^*} |\widehat y|^2\dxt 
+ \iint\limits_{Q}e^{3s\beta^*} |\widehat \theta|^2\dxt 
+ \int\limits_0^T\int\limits_{\om} e^{2s\widehat{\beta}+3s\beta^*}\widehat \gamma^{-7} |\widehat v_1|^2\dxt \\
+\int\limits_0^T\int\limits_{\om} e^{4s\widehat{\beta}+s\beta^*}\widehat \gamma^{-49/4} |\widehat v_0|^2\dxt = a((\widehat \chi,\widehat \sigma,\widehat\kappa),(\widehat \chi,\widehat \sigma,\kappa))< +\infty
\end{multline*}
and also that $(\widehat y,\widehat\theta)$ is, together with some pressure $\widehat p$, the weak solution of the system \eqref{eq:linearBouss2} for $v_1=\widehat v_1$ and $v_0=\widehat v_0$.

It only remains to check that 
$$e^{3/2s\beta^*}(\gamma^*)^{-9/8}\widehat y\in L^2(0,T;H^2(\Om)^3)\cap
L^{\infty}(0,T;V)$$
and
$$e^{3/2s\beta^*}(\gamma^*)^{-9/8}\widehat \theta\in L^2(0,T;H^2(\Om))\cap
L^{\infty}(0,T;H^1_0(\Om))$$
To this end, we define the functions
$$y^*=e^{3/2 s\beta^*}(\gamma^*)^{-9/8}\,\widehat y,\,
p^*=e^{3/2 s\beta^*}(\gamma^*)^{-9/8}\,\widehat p,\,
\theta^*=e^{3/2 s\beta^*}(\gamma^*)^{-9/8}\,\widehat \theta$$
$$f^*=e^{3/2 s\beta^*}(\gamma^*)^{-9/8}(f+(\widehat v_1,0,0)\1_{ \om})\mbox{ and }
f_0^*=e^{3/2 s\beta^*}(\gamma^*)^{-9/8}(f_0+\widehat v_0\1_{ \om}).$$ 
Then
$(y^*,p^*,\theta^*)$ satisfies 
\begin{equation*}
\left\lbrace \begin{array}{ll} 
    L y^*  + \nabla p^* = f^*  + \theta^*\,e_3 + (e^{3/2 s\beta^*}(\gamma^*)^{-9/8})_t \widehat y & \mbox{ in }Q, \\
    L \theta^* + y^*\cdot \nabla\bar\theta = f_0^* + (e^{3/2 s\beta^*}(\gamma^*)^{-9/8})_t \widehat \theta  & \mbox{ in }Q, \\
    \nabla\cdot y^* = 0 & \mbox{ in }Q, \\
    y^* = 0,\, \theta^*=0 & \mbox{ on }\Sigma, \\
    y^*(0) = e^{3/2 s\beta^*(0)}(\gamma^*(0))^{-9/8}y^0,\,  & \mbox{ in }\Om,\\
\theta^*(0)=e^{3/2 s\beta^*(0)}(\gamma^*(0))^{-9/8}\theta^0,\,  & \mbox{ in }\Om.
 \end{array}\right.
\end{equation*}
From the fact that $f^*+(e^{3/2 s\beta^*}(\gamma^*)^{-9/8})_t\,\widehat y \in L^2(Q)^3$, $f_0^* + (e^{3/2 s\beta^*}(\gamma^*)^{-9/8})_t \widehat\theta \in L^2(Q)$, $y^0\in V$ and $\theta^0\in H^1_0(\Om)$, we have indeed
$$y^*\in L^2(0,T;H^2(\Om)^3)\cap L^{\infty}(0,T;V) \mbox{ and }
 \theta^*\in L^2(0,T;H^2(\Om))\cap L^{\infty}(0,T;H^1_0(\Om))  $$
(see \eqref{eq:regularity1}). This ends the sketch of the proof of Proposition \ref{prop:null}.
\end{proof}
%
\section{Proof of Theorem \ref{teo:controlBouss}}
 
In this section we give the proof of Theorem \ref{teo:controlBouss} using similar arguments to those in \cite{OlegN-S} (see also \cite{E&S&O&P}, \cite{E&S&O&P-N-1}, \cite{Gue} and \cite{CarrGue}). The result of null controllability for the linear system \eqref{eq:linearBouss2} given by Proposition \ref{prop:null} will allow us to apply an inverse mapping theorem. Namely, we will use the following theorem (see \cite{ATF}).

\begin{theorem}\label{teo:invmap}
Let $B_1$ and $B_2$ be two Banach spaces and let $\mathcal{A}:B_1 \to B_2$ satisfy $\mathcal{A}\in C^1(B_1;B_2)$. Assume that $b_1\in B_1$, $\mathcal{A}(b_1)=b_2$ and that $\mathcal{A}'(b_1):B_1 \to B_2$ is surjective. Then, there exists $\delta >0$ such that, for every $b'\in B_2$ satisfying $\|b'-b_2\|_{B_2}< \delta$, there exists a solution of the equation
$$\mathcal{A} (b) = b',\quad b\in B_1.$$
\end{theorem}

Let us set
$$y=\widetilde y,\,  p = \bar p + \widetilde p \mbox{ and }\theta = \bar\theta + \widetilde\theta.$$
Using \eqref{eq:Bouss} and \eqref{eq:ncBouss0} we obtain
\begin{equation}\label{eq:Bouss2}
\left\lbrace \begin{array}{ll}
    \widetilde y_t - \Delta \widetilde y + (\widetilde y\cdot \nabla)\widetilde y + \nabla\widetilde p = v\1_{\om} +\widetilde \theta\,e_N& \mbox{ in }Q, \\
    \widetilde \theta_t - \Delta\widetilde \theta +\widetilde y\cdot \nabla\widetilde\theta +\widetilde y\cdot \nabla\bar\theta= v_0\1_{\om} & \mbox{ in }Q, \\
    \nabla\cdot\widetilde y = 0 & \mbox{ in }Q, \\
    \widetilde y = 0,\,\widetilde \theta=0 & \mbox{ on }\Sigma, \\
    \widetilde y(0) = y^0,\,\widetilde \theta(0)= \theta^0 - \bar\theta^0 & \mbox{ in }\Om.
 \end{array}\right.
\end{equation}
Thus, we have reduced our problem to the local null controllability of the nonlinear system \eqref{eq:Bouss2}.

We apply Theorem \ref{teo:invmap} setting
 $$B_1 = E,$$ $$B_2 = L^2(e^{5/2 s\beta^*}(\gamma^*)^{-2}(0,T);L^2(\Om)^3) \times V
\times L^2(e^{5/2 s\beta^*}(\gamma^*)^{-5/2}(0,T);L^2(\Om)) \times H^1_0(\Om)$$ 
and the operator
\begin{equation*}
\begin{split}
\mathcal{A}(\widetilde y,\widetilde p,v_1,\widetilde \theta,v_0) = (L\widetilde y + (\widetilde y\cdot \nabla)\widetilde y +\nabla \widetilde p - \widetilde \theta\,e_3 - (v_1,0,0)\1_{\om}&,\,
\widetilde y(0),\\
L \widetilde \theta + \widetilde y\cdot \nabla \widetilde \theta + \widetilde y\cdot\nabla\bar\theta - v_0\1_{\om}&,\,
\widetilde \theta(0))
\end{split}
\end{equation*}
for $(\widetilde y,\widetilde p,v_1,\widetilde \theta,v_0)\in E$.

In order to apply Theorem \ref{teo:invmap}, it remains to check that the operator $\mathcal{A}$ is of class $C^1(B_1;B_2)$. Indeed, notice that all the terms in $\mathcal{A}$ are linear, except for $(\widetilde y\cdot \nabla)\widetilde y$ and $\widetilde y\cdot \nabla \widetilde \theta$. We will prove that the bilinear operator
$$((y^1,p^1,v^1_1,\theta^1,v^1_0),(y^2,p^2,v^2_1,\theta^2,v^2_0))\to(y^1\cdot \nabla)y^2$$
is continuous from $B_1\times B_1$ to $ L^2(e^{5/2 s\beta^*}(\gamma^*)^{-2}(0,T);L^2(\Om)^3)$. To do this, notice that $$e^{3/2 s\beta^*}(\gamma^*)^{-9/8}y \in L^2(0,T;H^2(\Om)^3)\cap L^{\infty}(0,T;V)$$
 for any $(y,p,v_1,\theta,v_0)\in B_1$, so we have
$$e^{3/2 s\beta^*}(\gamma^*)^{-9/8}y \in L^2(0,T;L^{\infty}(\Om)^3)$$
and  
$$\nabla (e^{3/2 s\beta^*}(\gamma^*)^{-9/8}y) \in L^{\infty}(0,T;L^2(\Om)^3).$$
Consequently, we obtain
\begin{equation*}
\begin{split}
&\|e^{5/2 s\beta^*}(\gamma^*)^{-2}(y^1\cdot \nabla)y^2\|_{L^2(Q)^3} \\ 
&\leq C  \|(e^{3/2 s\beta^*}(\gamma^*)^{-9/8}\,y^1\cdot \nabla)e^{3/2 s\beta^*}(\gamma^*)^{-9/8}\,y^2\|_{L^2(Q)^3} \\
&\leq C \|e^{3/2 s\beta^*}(\gamma^*)^{-9/8}y^1\|_{L^2(0,T;L^{\infty}(\Om)^3)}\, \|e^{3/2 s\beta^*}(\gamma^*)^{-9/8}y^2\|_{L^{\infty}(0,T;V)}.
\end{split} 
\end{equation*}

In the same way, we can prove that the bilinear operator
$$((y^1,p^1,v^1_1,\theta^1,v^1_0),(y^2,p^2,v^2_1,\theta^2,v^2_0))\to y^1\cdot \nabla\theta^2$$
is continuous from $B_1\times B_1$ to $ L^2(e^{5/2 s\beta^*}(\gamma^*)^{-5/2}(0,T);L^2(\Om))$ just by taking into account that 
$$ e^{3/2 s\beta^*}(\gamma^*)^{-9/8}\theta \in L^{\infty}(0,T;H^1_0(\Om)),$$
 for any $(y,p,v_1,\theta,v_0)\in B_1$.

\vskip.5cm

Notice that $\mathcal{A}'(0,0,0,0,0):B_1\to B_2$ is given by
\begin{equation*}
\begin{split}\mathcal{A}'(0,0,0,0,0)(\widetilde y,\widetilde p,v_1,\widetilde \theta,v_0) = (L\widetilde y +\nabla \widetilde p - \widetilde \theta\,e_3 - (v_1,0,0)\1_{\om}&,\,
\widetilde y(0),\\
L \widetilde \theta + \widetilde y\cdot\nabla\bar\theta - v_0\1_{\om}&,\,
\widetilde \theta(0)),
\end{split}
\end{equation*}
for all $(\widetilde y,\widetilde p,v_1,\widetilde \theta,v_0)\in B_1,$
so this functional is surjective in view of the null controllability result for the linear system \eqref{eq:linearBouss2} given by Proposition \ref{prop:null}.

We are now able to apply Theorem \ref{teo:invmap} for $b_1=(0,0,0,0,0)$ and $b_2=(0,0,0,0)$. In particular, this gives  the existence of a positive number $\delta>0$ such that, if $\|\widetilde y(0),\widetilde \theta(0)\|_{V\times H^1_0(\Om)}\leq \delta$, then we can find controls $v_1$ and $v_0$ such that the associated solution $(\widetilde y,\widetilde p,\widetilde \theta)$ to \eqref{eq:Bouss2} satisfies $\widetilde y(T)=0$ and $\widetilde \theta(T)=0$ in $\Om$.

This concludes the proof of Theorem \ref{teo:controlBouss}.

\section*{Acknowledgments} The author would like to thank the ``Agence Nationale de la Recherche'' (ANR), Project CISIFS, grant ANR-09-BLAN-0213-02, for partially supporting this work.

\end{document}